\documentclass[11pt,final]{iopart}
\usepackage{amssymb}
\usepackage{amsfonts}
\usepackage{amsthm}
\usepackage{iopams}
\usepackage[dvips]{epsfig}
\usepackage{float}
\usepackage{footmisc}
\usepackage[title,titletoc,toc]{appendix}
\usepackage{prettyref}
\usepackage{appendix}
\usepackage{showkeys}
\usepackage{todo}
\usepackage{prettyref}
\usepackage{verbatim}
\usepackage[latin1]{inputenc}
\usepackage{subfigure}
\usepackage{verbatim}
\newrefformat{eq}{(\ref{#1})}
\newrefformat{fig}{Figure~\ref{#1}}
\usepackage[pdfpagelabels]{hyperref}
\hypersetup{colorlinks=true,
linkcolor=blue,
citecolor=blue,
plainpages=false}
\newtheorem{theorem}{\sc Theorem}[section]
\newtheorem{propn}[theorem]{\sc Proposition}

\newtheorem{lemma}[theorem]{\sc Lemma}

\newtheorem{definition}[theorem]{\sc Definition}
\newtheorem{remark}[theorem]{\sc Remark}

\newtheorem{cor}[theorem]{\sc Corollary}

\newcommand{\argmin}{{\rm arg\,min}}

\eqnobysec

\makeindex

\newcommand{\cB}{{\mathcal B}}
\newcommand{\cC}{{\mathcal C}}

\newcommand{\cH}{{\mathcal H}}

\newcommand{\cL}{{\mathcal L}}

\newcommand{\cO}{{\mathcal O}}
\newcommand{\cP}{{\mathcal P}}

\newcommand{\cR}{{\mathcal R}}

\newcommand{\cT}{{\mathcal T}}

\newcommand{\tu}{\tilde{u}}

\newcommand{\tL}{\tilde{L}}

\newcommand{\R}{\mathbb{R}}

\newcommand{\norm}[1]{|| #1||}

\newcommand{\ba}{\begin{array}}
\newcommand{\ea}{\end{array}}
\newcommand{\be}{\begin{equation}}
\newcommand{\ee}{\end{equation}}
\newcommand{\bea}{\begin{eqnarray}}
\newcommand{\eea}{\end{eqnarray}}
\newcommand{\beq}{\begin{equation}}
\newcommand{\eeq}{\end{equation}}
\newcommand{\bqt}{\begin{quote}}
\newcommand{\eqt}{\end{quote}}

\begin{document}

%
\title[Convergence analysis in convex regularization]
{Convergence analysis in convex regularization depending
on the smoothness degree of the penalizer}

%
\author{Erdem Altuntac}

\address{Institute for Numerical and Applied Mathematics,
University of G\"{o}ttingen, Lotzestr. 16-18,
D-37083, G\"{o}ttingen, Germany}


\ead{\mailto{e.altuntac@math.uni-goettingen.de}}

\begin{abstract}

The problem of minimization of the least squares 
functional with a smooth, lower semi-continuous, convex penalizer 
$J(\cdot)$ is considered to be solved. 
Over some compact and convex subset $\Omega$ of the Hilbert
space $\cH,$ the regularizer is implicitly defined as 
$ J(\cdot) : \cC^{k}(\Omega , \cH) \rightarrow \R_{+}$
where $k \in \{1,2\}.$ So the cost functional associated with
some given linear, compact and injective forward operator 
$\cT :\Omega \subset \cH \rightarrow \cH,$

\bea
F_{\alpha}(\cdot , f^{\delta}) := \frac{1}{2} \norm{\cT( \cdot ) - f^{\delta}}_{\cH}^2 + \alpha J(\cdot) ,
\nonumber
\eea
where $f^{\delta}$ is the given perturbed data with its
perturbation amount $\delta$ in it.
Convergence of the regularized optimum solution
$\varphi_{\alpha(\delta)} \in \argmin F_{\alpha}(\varphi , f^{\delta})$
to the true solution $\varphi^{\dagger}$ is analysed
depending on the smoothness degree of the penalizer, 
{\em i.e.} the cases $k \in \{1,2\}$ in  
$ J(\cdot) : \cC^{k}(\Omega , \cH) \rightarrow \R_{+}.$ 
In both cases, we define such a regularization parameter 
that is in cooperation with the condition

\bea
\alpha(\delta , f^{\delta}) \in \{ \alpha > 0 \mbox{ }\vert \mbox{ }\norm{\cT\varphi_{\alpha(\delta)} - f^{\delta}}\leq \tau\delta \} ,
\nonumber
\eea
for some fixed $\tau \geq 1.$ In the case of $k = 2,$ we 
are able to evaluate the discrepancy 
$\norm{\cT\varphi_{\alpha(\delta)} - f^{\delta}}\leq \tau\delta$
with the Hessian Lipschitz constant $L_H$ of the functional 
$F_{\alpha}(\cdot , f^{\delta}).$

\end{abstract}

\bigskip


\section{Introduction}

In this work, over some compact and convex subset $\Omega$
of the Hilbert space $\cH,$ we consider solving formulate our main 
variational minimization problem,

\beq
\label{problem0}
\argmin_{\Omega \subset \cH} \left\{
F_{\alpha}(\cdot , f^{\delta}) := 
\frac{1}{2} \norm{\cT( \cdot ) - f^{\delta}}_{\cH}^2 + \alpha J(\cdot) \right\} .
\eeq
Here, $ J(\cdot) : \cC^{k}(\Omega , \cH) \rightarrow \R_{+},$
for $k = \{1 , 2\}$ is convex and $\alpha > 0$ is the regularization parameter.
Following \cite{ColtonKress13, Engl96, Isakov06}, we construct
the parametrized solution $\varphi_{\alpha(\delta)}$ for the problem (\ref{problem0})
satisfying

\begin{enumerate}
\item For any $f \in \cH$ there exists a solution $\varphi_{\alpha}\in \cH$ to the problem (\ref{problem0});
\item For any $f \in \cH$ there is no more than one $\varphi_{\alpha}\in \cH;$
\item Convergence of the regularized solution $\varphi_{\alpha}$ to the
true solution $\varphi^{\dagger}$ must depend on the given data, {\em i.e.}
\begin{displaymath}
\norm{\varphi_{\alpha(\delta)} - \varphi^{\dagger}}_{\cH} \rightarrow 0 \mbox{ as } \alpha(\delta) \rightarrow 0 \mbox{ for } \delta \rightarrow 0
\end{displaymath}
whilst
\begin{displaymath}
\norm{f^{\dagger} - f^{\delta}} \leq \delta
\end{displaymath}
where $f^{\dagger} \in \cH$ is the true measurement
and $\delta$ is the noise level.
\end{enumerate}
What is stated by `(iii)' is that when the given measurement $f^{\delta}$
lies in some $\delta-$ball centered at the true measurement $f^{\dagger}$, 
$\cB_{\delta}(f^{\dagger}),$ then the expected solution must lie in
the corresponding $\alpha(\delta)$ ball. 
It is also required that this solution $\varphi_{\alpha(\delta)}$ 
must depend on the data $f^{\delta}.$ Therefore, we are always tasked with 
finding an approximation of the unbounded inverse operator
$\cT^{-1}: \cR(T) \rightarrow \cH$ by a bounded linear
operator $R_{\alpha} : \cH \rightarrow \cH.$

\begin{definition}[Regularization operator]
\cite[Definition 4.3]{ColtonKress13},\cite[Theorem 2.2]{Lorenz08}
\label{def_regularization_operator}
Let $\cT : \cH \rightarrow \cH$ be some given linear injective
operator. Then a family of bounded operators 
$R_{\alpha} : \cH \rightarrow \cH,$ $\alpha > 0,$ with the property
of pointwise convergence

\beq
\label{regularization_pointwise_convergence}
\lim_{\alpha \rightarrow 0 } R_{\alpha}\cT\varphi^{\dagger} = \varphi^{\dagger} 
\eeq
is called a regularization scheme for the operator $\cT.$ The parameter
$\alpha$ is called regularization paremeter.
\end{definition}

As alternative to well established Tikhonov regularization,
\cite{Tikhonov63, TikhonovArsenin77},
studying convex variational regularization with any penalizer $J(\cdot)$
has become important over the last decade. Introducing a new image 
denoising method named as {\em total variation}, \cite{RudenOsherFatemi92},
is commencement of this study.
Application and analysis of the method have been widely carried
out in the communities of inverse problems and optimization, 
\cite{AcarVogel94, BachmayrBurger09, BardsleyLuttman09, ChambolleLions97,
ChanChen06, ChanGolubMulet99, DobsonScherzer96, 
DobsonVogel97, VogelOman96}. Particularly, formulating
the minimization problem as variational problem and
estimating convergence rates with variational source conditions
has also become popular recently, \cite{BurgerOsher04, Grasmair10, 
Grasmair13, GrasmairHaltmeierScherzer11, Lorenz08}.
Different from available literature, we take into account
one fact; for some given measurement $f^{\delta}$ with the noise level 
$\delta$ and forward operator $\cT,$ the regularized solution $\varphi_{\alpha(\delta)}$ 
to the problem (\ref{problem0}) should satisfy 
$\norm{\cT\varphi_{\alpha(\delta)} - f^{\delta}} \leq \tau\delta$
for some fixed $\tau \geq 1.$ With this fact, we manage to obtain
tight convergence rates for $\norm{\varphi_{\alpha(\delta)} - \varphi^{\dagger}},$
and we can carry out this analysis for a general smooth, convex
penalty $J(\cdot) \in \cC^{k}(\Omega,\cH)$ 
for the cases $k = \{1, 2\}.$
We will be able to quantify the tight convergence rates 
under the assumption that $J(\cdot)$ is defined over $\cC^{k}(\Omega,\cH)$
space for $k \in \{1, 2\}.$ To be more specific, 
we will observe that rule for the choice of regularization 
paremeter $\alpha(\delta)$ must contain Lipschitz constant
in addition to the noise level $\delta.$
That is, when $k = 2,$ we will need $\cC^{2+}$ class. 

\section{Notations and prerequisite knowledge}
\label{notations}

Let $\cC(\Omega)$ be the space of continuous functions on
the compact domain $\Omega.$ Then, $\cC^{k}(\Omega)$ 
function space

\begin{displaymath}
\cC^{k}(\Omega) := \{ \varphi \in \cC(\Omega) : \nabla^{(k)}\varphi \in \cC(\Omega)\} .
\end{displaymath}
Addition to traditional $\cC^{k}$ spaces, we will need to address
$\cC^{k+}$ for the purpose of convergence analysis.
In general for an open set $O \subset \R^{N},$ a mapping
$\cP : O \rightarrow \R^{N}$ is said to be of {\em class} $\cC^{k+}$
if it is of class $\cC^{k}$ and $k$th partial derivatives
are not just continuous but strictly continuous on $O,$ 
\cite[pp. 355]{RockafellarWets98}. Then, for a
smooth and convex functional $J(\varphi)$ defined over
$\cC^{k}(\Omega,\cH),$ there exists Lipschitz constant $\tL$
such that 

\beq
\label{Ck_plus_space}
\norm{\nabla^{(k)}J(\varphi) - \nabla^{(k)}J(\Psi)} \leq \tL \norm{\varphi - \Psi} .
\eeq
When $k = 1,$ by $\tL$ we denote well-known Lipschitz constant $L$. 
When $k = 2,$ $\tL$ will be Hessian Lipschitz $L_H$, \cite{FowkesGouldFarmer13}.

Over some compact and convex domain $\Omega \subset \cH,$
variational minimization problem is formulated as such,

\beq
\label{problem2}
\argmin_{\varphi \in \cH} \left\{
F_{\alpha}(\cdot , f^{\delta}) := 
\frac{1}{2} \norm{\cT( \cdot ) - f^{\delta}}_{\cH}^2 + \alpha J(\cdot) \right\}
\eeq
with its penalty 
$ J(\cdot) : \cC^{k}(\Omega , \cH) \rightarrow \R_{+},$
where $k = \{1 , 2\},$
and $\alpha > 0$ is the regularization parameter.
Another dual minimization problem to (\ref{problem2}) is given by 

\beq
\label{constrained_problem}
J(\cdot) \rightarrow \min_{\cH}
\mbox{, subject to } \norm{\cT(\cdot) - f^{\delta}} \leq \delta .
\eeq
In the Hilbert scales, it is known that the solution of 
the penalized minimizatin problem (\ref{problem2}) equals to the 
solution of the constrained minimization problem 
(\ref{constrained_problem}), \cite[Subsection 3.1]{BurgerOsher04}.
The regularized solution $\varphi_{\alpha(\delta)}$ 
of the problem (\ref{problem2}) satisfies the following 
first order optimality conditions,

\bea
\label{optimality_1}
& 0 & = \nabla F_{\alpha}(\varphi_{\alpha(\delta)})
\nonumber\\
& 0 & = \cT^{\ast}(\cT\varphi_{\alpha(\delta)} - f^{\delta}) + \alpha(\delta) \nabla J(\varphi_{\alpha(\delta)})
\nonumber\\
& \cT^{\ast}(f^{\delta} - \cT\varphi_{\alpha(\delta)}) & = \alpha(\delta) \nabla  J(\varphi_{\alpha(\delta)}) .
\eea

In this work, the radii $\delta$ of the $\alpha(\delta)$ ball are estimated,
by means of the Bregman divergence, 
with potential 
$J(\cdot) : \cC^{1}(\Omega , \cH) \rightarrow \R_{+}.$
The choice of regularization parameter $\alpha(\delta)$ 
in this work does not require any {\em a priori} knowledge about 
the true solution. We always work with perturbed data
$f^{\delta}$ and introduce the rates according 
to the perturbation amount $\delta.$


\subsection{Bregman divergence}
\label{bregman_divergence_def}

We will be able to quantify
the rate of the convergence of $\norm{\varphi_{\alpha(\delta)} - \varphi^{\dagger}}$
by means of different formulations of the Bregman divergence.
Following formulation emphasizes the functionality of the Bregman
divergence in proving the norm convergence of the minimizer
of the convex minimization problem to the true solution.

\begin{definition}[Total convexity and Bregman divergence]\cite[Def.1]{Bredies09}
\label{total_convexity}

Let $\Phi : \cH \rightarrow \R \cup \{ \infty \}$ be a smooth and convex functional. 
Then $\Phi$ is called totally convex in $u^{\ast} \in \cH,$ 
if, for $\nabla \Phi(u^{\ast})$ and $\{ u \},$ it holds that

\bea
D_{\Phi}(u , u^{\ast}) =
\Phi(u) - \Phi(u^{\ast}) - \langle \nabla \Phi(u^{\ast}) , u - u^{\ast} \rangle \rightarrow 0
\Rightarrow \norm{u - u^{\ast}}_{\cH} \rightarrow 0
\nonumber
\eea
where $D_{\Phi}(u , u^{\ast})$ represents the {\em Bregman divergence}.

It is said that $\Phi$ is {\em q-convex} in $u^{\ast} \in \cH$
with a $q \in [2, \infty ),$ if for all $M > 0$ there exists a $c^{\ast} > 0$ 
such that for all $\norm{u - u^{\ast}}_{\cH} \leq M$ we have

\beq
\label{q_convexity}
D_{\Phi}(u , u^{\ast}) = 
\Phi(u) - \Phi(u^{\ast}) - \langle \nabla \Phi(u^{\ast}), u - u^{\ast} \rangle \geq c^{\ast} \norm{u - u^{\ast}}_{\cH}^q .
\eeq
\end{definition}

Throughout our norm convergence estimations, we refer to this
definition for the case of $2-$convexity. We will also study
different formulations of the Bregman divergence. We introduce
these different formulations below.

\begin{remark}[Different formulations of the Bregman divergence]
\label{various_bregman}
Let $\varphi_{\alpha(\delta)} , \varphi^{\dagger}$ defined on $\Omega$
respectively be the regularized and the true solutions of the problem
(\ref{problem2}). Then we give the following
definitions of the Bregman divergence;

\begin{itemize}
\item Bregman distance associated with the cost functional $F(\cdot):$
\beq
\label{bregman_def_cost_func}
D_{F}(\varphi_{\alpha(\delta)} , \varphi^{\dagger})
= F(\varphi_{\alpha(\delta)}) - F(\varphi^{\dagger}) 
- \langle \nabla F(\varphi^{\dagger}) , \varphi_{\alpha(\delta)} - \varphi^{\dagger} \rangle ,
\eeq

\item Bregman distance associated with the penalty $J(\cdot):$
\beq
\label{bregman_def_regularizer}
D_{J}(\varphi_{\alpha(\delta)} , \varphi^{\dagger})
= J(\varphi_{\alpha(\delta)}) - J(\varphi^{\dagger})
- \langle \nabla J(\varphi^{\dagger}) , \varphi_{\alpha(\delta)} - \varphi^{\dagger} \rangle
\eeq

\item Bregman distance associated with the misfit term 
$G_{\delta}(\cdot , f^{\delta}) := \frac{1}{2}\norm{\cT(\cdot) - f^{\delta}}^2:$
\beq
\label{bregman_def_misfit}
D_{G_{\delta}}(\varphi_{\alpha(\delta)} , \varphi^{\dagger})
= \frac{1}{2}\norm{\cT\varphi_{\alpha(\delta)} - f^{\delta}}^2 - \frac{1}{2}\norm{\cT\varphi^{\dagger} - f^{\delta}}^2
- \langle \nabla G_{\delta}(\varphi^{\dagger} , f^{\delta}) , \varphi_{\alpha(\delta)} - \varphi^{\dagger} \rangle
\eeq
\end{itemize}
\end{remark}
Reader may also refer to Appendix \ref{more_Bregman_divergence}
for further properties of the Bregman divergence.
In fact, another similar estimation 
to (\ref{q_convexity}), for $q = 2,$ can also be derived
by making further assumption about the functional $\Phi$
one of which is strong convexity with modulus $c,$
\cite[Definition 10.5]{BauschkeCombettes11}.
Below is this alternative way of obtaining (\ref{q_convexity})
when $q = 2.$

\begin{propn}
\label{proposition_q-convexity}
Let $\Phi : \cH \rightarrow \R \cup \{ \infty \}$ 
be $\Phi \in \cC^{2}(\cH)$ is strongly convex with 
modulus of convexity $c > 0,$ {\em i.e.} $\nabla^2 \Phi \succ cI,$ then

\bea
\label{lower_bound_for_bregman}
D_{\Phi}(u , v) > c \norm{u - v}^2 + \cO(\norm{u - v}^2) .
\eea
\end{propn}

\begin{proof}

Let us begin with considering the Taylor expansion of $\Phi,$
\beq
\Phi(u) = \Phi(v) + \langle \nabla \Phi(v) , u - v \rangle + 
\frac{1}{2} \langle \nabla^2 \Phi(v)(u - v) , u - v \rangle + \cO(\norm{u - v}^2) .
\eeq
Then the Bregman divergence 

\bea
D_{\Phi}(u , v) & = & 
\Phi(u) - \Phi(v) - \langle \nabla \Phi(v) , u - v \rangle
\nonumber\\
& = & \langle \nabla \Phi(v) , u - v \rangle + 
\frac{1}{2} \langle \nabla^2 \Phi(v)(u - v) , u - v \rangle + \cO(\norm{u - v}^2) 
- \langle \nabla \Phi(v) , u - v \rangle
\nonumber\\
& = & \frac{1}{2} \langle \nabla^2 \Phi(v)(u - v) , 
u - v \rangle + \cO(\norm{u - v}^2) .
\nonumber
\eea
Since $\Phi(\cdot)$ is striclty convex, due to strong convexity 
and $\Phi \in \cC^{2}(\cH),$ hence one obtains that

\bea
D_{\Phi}(u , v) > c \norm{u - v}^2 + \cO(\norm{u - v}^2) ,
\eea
where $c$ is the modulus of convexity.

\end{proof}

Above, in (\ref{bregman_def_misfit}), we have set $\Phi := G_{\delta}(\cdot , f^{\delta}).$
In this case, one must assume even more than stated about the existence of the modulus of
convexity $c.$ These assumptions can be formulated in the following
way. Suppose that there exists some measurement $f^{\delta}$ lying
in the $\delta-$ball $\cB_{\delta}(f^{\dagger})$ for all $\delta > 0$
small enough such that the followings hold,

\bea 
\label{misfit_2-cnvx_assumption1}
0 < c_{\delta} & \leq & c_{f^{\delta}} ,
\\
\label{misfit_2-cnvx_assumption2}
0 < \underline{c} & \leq & c_{\delta}, \mbox{ for all } \delta > 0.
\eea 
Then $G_{\delta}(\cdot , f^{\delta})$ is $2-$convex and
according to Proposition \ref{proposition_q-convexity},

\bea
D_{G_{\delta}}(\varphi_{\alpha(\delta)} , \varphi^{\dagger}) 
> c_{f^{\delta}} \norm{\varphi_{\alpha(\delta)} - \varphi^{\dagger}}^2 + 
\cO(\norm{\varphi_{\alpha(\delta)} - \varphi^{\dagger}}^2) ,
\eea

Addition to the traditional definition of Bregman divergence in (\ref{q_convexity}),
{\em symmetrical Bregman divergence} is also given below, \cite[Definition 2.1]{Grasmair13},

\beq
\label{symmetrical_bregman}
D_{\Phi}^{\mbox{sym}}(u , u^{\ast}) := D_{\Phi}(u , u^{\ast}) + D_{\Phi}(u^{\ast} , u) .
\eeq
With symmetrical Bregman divergence having formulated,
following from the Definition \ref{total_convexity}, we give
the last proposition for this chapter.

\begin{propn}\cite[as appears in the proof of Theorem 4.4]{Grasmair13}
Let $\Phi : \cH \rightarrow \R \cup \{ \infty \}$ be a smooth and q-convex functional. 
Then there exist positive constants $c^{\ast}, c > 0$ 
such that for all $\norm{u - u^{\ast}}_{\cH} \leq M$ we have

\bea
\label{sym_bregman_to_norm_conv}
D_{\Phi}^{\mbox{sym}}(u , u^{\ast}) &=& \langle \nabla \Phi(u^{\ast}) - \nabla \Phi(\tu) , u - u^{\ast} \rangle
\nonumber\\
& \geq & (c^{\ast} + c) \norm{u - u^{\ast}}_{\cH}^2 .
\eea
\end{propn}

\begin{proof}
Proof is a straightforward result of the estimation in (\ref{q_convexity})
and the symmetrical Bregman divergence definition given by (\ref{symmetrical_bregman}).
\end{proof}

\subsection{Appropriate regularization parameter with discrepancy principle}
\label{choice_of_regpar}

A regularization parameter $\alpha$ is admissible
for $\delta$ when

\beq
\label{discrepancy}
\norm{\cT\varphi_{\alpha} - f^{\delta}}\leq \tau\delta
\eeq
for some fixed $\tau \geq 1.$ We seek a rule for chosing
$\alpha(\delta)$ as a function of $\delta$
such that (\ref{discrepancy}) is satisfied and

\begin{displaymath}
\alpha(\delta) \rightarrow 0, \mbox{ as } \delta \rightarrow 0 .
\end{displaymath}
Folllowing \cite[Eq. (4.57) and (4.58)]{Engl96},
\cite[Definition 2.3]{Kirsch11},
in order to obtain tight rates of convergence of 
$\norm{\varphi_{\alpha(\delta)} - \varphi^{\dagger}}$
we define $\alpha(\delta , f^{\delta})$ such that

\beq
\label{discrepancy_pr_definition}
\alpha(\delta , f^{\delta}) \in \{ \alpha > 0 \mbox{ }\vert \mbox{ }
\norm{\cT\varphi_{\alpha} - f^{\delta}}\leq \tau\delta , \mbox{ for all given } (\delta , f^{\delta}) \} .
\eeq
The strong
relation between the discrepancy $\norm{\cT\varphi_{\alpha(\delta)} - f^{\delta}}$
and the norm convergence of $\norm{\varphi_{\alpha(\delta)} - \varphi^{\dagger}}$
can be formulated in the following lemma.

\begin{lemma}
\label{general_discrepancy_lemma}
Let $\cT : \cH \rightarrow \cH$ be a linear and compact operator.
Denote by $\varphi_{\alpha(\delta)}$ the regularized solution
and by $\varphi^{\dagger}$ the true solution 
to the problem (\ref{problem2}). Then 

\beq
\label{residual1}
\norm{\cT\varphi_{\alpha(\delta)} - f^{\delta}} \leq 
\delta + \norm{\varphi_{\alpha(\delta)} - \varphi^{\dagger}} \norm{\cT^{\ast}} ,
\eeq
where the noisy data $f^{\delta}$ to the true
data $f^{\dagger}$ both satisfy
$\norm{f^{\delta} - f^{\dagger}} \leq \delta$ 
for sufficiently small amount of noise $\delta.$
\end{lemma}

\begin{proof}
Desired result follows from the following straightforward
calculations,

\bea
\norm{\cT\varphi_{\alpha(\delta)} - f^{\delta}}^2 & = & 
\langle \cT\varphi_{\alpha(\delta)} - f^{\delta} , \cT\varphi_{\alpha(\delta)} - f^{\delta} \rangle
\nonumber\\
& = & \langle \cT\varphi_{\alpha(\delta)} - f^{\dagger} + f^{\dagger} - f^{\delta} , \cT\varphi_{\alpha(\delta)} - f^{\delta} \rangle
\nonumber\\
& = & \langle \cT\varphi_{\alpha(\delta)} - f^{\dagger} , \cT\varphi_{\alpha(\delta)} - f^{\delta} \rangle 
+ \langle f^{\dagger} - f^{\delta} , \cT\varphi_{\alpha(\delta)} - f^{\delta} \rangle
\nonumber\\
& = & \langle \cT(\varphi_{\alpha(\delta)} - \varphi^{\dagger}) , \cT\varphi_{\alpha(\delta)} - f^{\delta} \rangle 
+ \langle f^{\dagger} - f^{\delta} , \cT\varphi_{\alpha(\delta)} - f^{\delta} \rangle
\nonumber\\
& = & \langle \varphi_{\alpha(\delta)} - \varphi^{\dagger} , \cT^{\ast}(\cT\varphi_{\alpha(\delta)} - f^{\delta}) \rangle
+ \langle f^{\dagger} - f^{\delta} , \cT\varphi_{\alpha(\delta)} - f^{\delta} \rangle
\nonumber\\
& \leq & \norm{\varphi_{\alpha(\delta)} - \varphi^{\dagger}} \norm{\cT^{\ast}} \norm{\cT\varphi_{\alpha(\delta)} - f^{\delta}}
+ \delta \norm{\cT\varphi_{\alpha(\delta)} - f^{\delta}} .
\nonumber
\eea

\end{proof}

\section{Monotonicity of the gradient of convex functionals}
\label{monotonicity}

If the positive real valued convex functional 
$ \cP(\cdot) : \cC^{1}(\Omega , \cH) \rightarrow \R_{+},$ 
is in the class of $\cC^{1},$ then for all 
$\varphi , \Psi$ defined on $\Omega \subset \cH,$
\beq
\label{convex_funct_gradient}
\cP(\Psi) \geq \cP(\varphi) + \langle \nabla  \cP(\varphi) , \Psi - \varphi \rangle .
\eeq
What this inequality basically means is that at each $\varphi$
the tangent line of the functional lies below the functional itself.
The same is also true from subdifferentiability point of view.
Following from (\ref{convex_funct_gradient}), one can also write that
\beq
\label{upper_boundary}
\cP(\varphi) - \cP(\Psi) \leq  \langle \nabla \cP(\varphi) , \varphi - \Psi \rangle .
\eeq
Still from (\ref{convex_funct_gradient}), by replacing $\varphi$ with $\Psi$ one obtains 
\beq
\cP(\varphi) \geq \cP(\Psi) + \langle \nabla \cP(\Psi) , \varphi - \Psi \rangle ,
\eeq
or equivalently
\beq
\label{lower_boundary}
\cP(\varphi) - \cP(\Psi) \geq \langle \nabla \cP(\Psi) , \varphi - \Psi \rangle .
\eeq
Combining (\ref{upper_boundary}) and (\ref{lower_boundary}) brings us,
\beq
\label{whole_boundary}
\langle \nabla \cP(\Psi) , \varphi - \Psi \rangle \leq
\cP(\varphi) - \cP(\Psi) \leq \langle \nabla \cP(\varphi) , \varphi - \Psi \rangle .
\eeq
Eventually this implies

\beq
\label{monotonicity_sub_diff}
0 \leq \langle \nabla \cP(\varphi) - \nabla \cP(\Psi) , \varphi - \Psi \rangle
\eeq
which is the monotonicity of the gradient of convex functionals,
\cite[Proposition 17.10]{BauschkeCombettes11}.

Initially, owing to the relation in (\ref{whole_boundary}), 
it can easily be shown the weak convergence of the 
regularized solution $\varphi_{\alpha(\delta)}$ 
to the true solution $\varphi^{\dagger}$, with the choice of 
regularization parameter $\alpha(\delta).$

\begin{theorem}[Weak convergence of the regularized solution]
\label{weak_convergence}

In the same conditions of Lemma \ref{general_discrepancy_lemma},
if the regularized minimum $\varphi_{\alpha(\delta)}$ 
to the problem (\ref{problem2}) exists and 
$\norm{f^{\delta} - f^{\dagger}}_{\cL^2} \leq \delta,$ then 

\beq
\varphi_{\alpha(\delta)} \rightharpoonup \varphi^{\dagger} \mbox{, as $\alpha(\delta) = \delta^{p} \rightarrow 0 $ for any $p \in (0,2).$}
\eeq

\end{theorem}

\begin{proof}
Since $\varphi_{\alpha(\delta)}$ is the minimizer 
of the cost functional 
$F(\varphi , f^{\delta}) : \cH \rightarrow \R_{+},$ then

\bea
F(\varphi_{\alpha(\delta)} , f^{\delta}) = \frac{1}{2}\norm{\cT \varphi_{\alpha(\delta)} - f^{\delta}}_{\cL^2}^2 + \alpha J(\varphi_{\alpha(\delta)})
\leq \frac{1}{2}\norm{\cT \varphi^{\dagger} - f^{\delta}}_{\cL^2}^2 + \alpha J(\varphi^{\dagger}) = F(\varphi^{\dagger} , f^{\delta}),
\nonumber
\eea
which is in other words,

\beq
\alpha (J(\varphi_{\alpha(\delta)}) - J(\varphi^{\dagger}))
\leq \frac{1}{2} \norm{\cT \varphi^{\dagger} - f^{\delta}}_{\cL^2}^2 - 
\frac{1}{2} \norm{\cT \varphi_{\alpha(\delta)} - f^{\delta}}_{\cL^2}^2 .
\eeq
From the convexity of the penalization term $J(\cdot),$ 
a lower boundary has been already found in (\ref{whole_boundary}).
Then following from (\ref{whole_boundary}), the last inequality implies,

\beq
\alpha \langle \nabla J(\varphi^{\dagger}) , \varphi_{\alpha(\delta)} - \varphi^{\dagger} \rangle
\leq \alpha (J(\varphi_{\alpha(\delta)}) - J(\varphi^{\dagger})) \leq \frac{1}{2} \delta^2 ,
\eeq
since $\norm{f^{\dagger} - f^{\delta}}_{\cL^2} \leq \delta.$ 
With the choice of $\alpha(\delta) = \delta^{p}$ 
for any $p \in (0,2),$ desired result is obtained

\beq
\langle \nabla J(\varphi^{\dagger}) , \varphi_{\alpha(\delta)} - \varphi^{\dagger} \rangle
\leq \frac{1}{2} \delta^{p-2} .
\eeq
\end{proof}

\begin{remark}
Note that the result of the theorem is true
for any smooth and convex penalty $J(\cdot)$ in the problem (\ref{problem2}).
\end{remark}


\section{Convergence Results for $\norm{\varphi_{\alpha(\delta)} - \varphi^{\dagger}}$}
 \label{convex_convergence}
We now come to the point where we analyse each cases when
$J(\cdot) : \cC^{k}(\Omega, \cH) \rightarrow \R_{+}$ for
$k \in \{1 ,2\}.$ In each case, we will consider the 
discrepancy principle for the choice of regularization
parameter while providing the norm convergence.
 
 
\subsection{When the penalty $J(\cdot)$ is defined over $\cC^{1}(\Omega,\cH)$}
\label{convergence}

First part of the following formulation has been studied in 
\cite[Theorem 5.]{BurgerOsher04}. 
There, the authors obtain some convergence 
in terms of a Lagrange multiplier 
$\lambda(\delta)$ instead of a regularization parameter $\alpha(\delta).$
According to theoretical set up given by the authors, their convergence
rate explicitly contain Lagrange multiplier defined as $\lambda(\delta) := 1/\delta.$
Second part, on the other hand, has been motivated by \cite[Theorem 4.4]{Grasmair13}.
All convergence results are obtained under the assumption
that the penalizer is $2-$convex according to (\ref{q_convexity}).

\begin{theorem}[Upper bound for the Bregman divergence associated with the penalty]
\label{Bregman_regularizer_upper_bound}
Let $J(\cdot) : \cC^{1}(\Omega,\cH) \rightarrow \R_{+}$
be the smooth and $2-$convex penalization term of the cost functional $F(\cdot , f)$
given in the problem (\ref{problem2}), and denote by 
$\varphi_{\alpha(\delta)}$ the regularizd solution of the same problem
satisfying $\norm{\cT\varphi_{\alpha(\delta)} - f^{\delta}} \leq \tau \delta$
where $\tau \geq 1$ as in (\ref{choice_of_regpar}).
Then, the choice of regularization parameter 
$\alpha(\delta) := \sqrt{\delta} (\tau + 1) \norm{\cT^{\ast}}$ yields,

\beq
D_{J}(\varphi_{\alpha(\delta)} , \varphi^{\dagger}) \leq \sqrt{\delta} \norm{\varphi_{\alpha(\delta)} - \varphi^{\dagger}},
\eeq
and
\beq
D_{J}^{\mbox{sym}}(\varphi_{\alpha(\delta)} , \varphi^{\dagger}) \leq \sqrt{\delta}\norm{\varphi_{\alpha(\delta)} - \varphi^{\dagger}} ,
\eeq
both of which imply,

\beq
\norm{\varphi_{\alpha(\delta)} - \varphi^{\dagger}} \leq \sqrt{\delta} .
\eeq
\end{theorem}

\begin{proof} 
First recall the formulation for the Bregman divergence
associated with the penalty $J(\cdot)$ in (\ref{bregman_def_regularizer}).
Convexity of the penalizer $J(\cdot)$ brings the following
estimation by the second part of (\ref{whole_boundary}),

\bea
J(\varphi_{\alpha(\delta)}) - J(\varphi^{\dagger})
\leq \langle \nabla J(\varphi_{\alpha(\delta)}) , \varphi_{\alpha(\delta)} - \varphi^{\dagger} \rangle
\nonumber
\eea
Then in fact (\ref{bregman_def_regularizer}) can be bounded by,

\bea
D_{J}(\varphi_{\alpha(\delta)} , \varphi^{\dagger}) & \leq &
\langle \nabla J(\varphi_{\alpha(\delta)}) - \nabla J(\varphi^{\dagger}), 
\varphi_{\alpha(\delta)} - \varphi^{\dagger} \rangle 
\nonumber\\
& = &\frac{1}{\alpha(\delta)} \langle \cT^{\ast}(f^{\delta} - \cT\varphi_{\alpha(\delta)}) - \cT^{\ast}(f^{\delta} - \cT\varphi^{\dagger}) , 
\varphi_{\alpha(\delta)} - \varphi^{\dagger} \rangle , 
\nonumber
\eea
due to the first order optimality conditions in (\ref{optimality_1}),
{\em i.e.} $\cT^{\ast}(f^{\delta} - \cT(\cdot) ) = \alpha(\delta) \nabla J(\cdot).$
The inner product can also be written in the composite form,

\bea
D_{J}(\varphi_{\alpha(\delta)} , \varphi^{\dagger}) \leq
\frac{1}{\alpha(\delta)} \langle \cT^{\ast}(f^{\delta} - \cT\varphi_{\alpha(\delta)}) ,
\varphi_{\alpha(\delta)} - \varphi^{\dagger} \rangle -
\frac{1}{\alpha(\delta)} \langle
\cT^{\ast}(f^{\delta} - f^{\dagger}) , 
\varphi_{\alpha(\delta)} - \varphi^{\dagger} \rangle ,
\nonumber
\eea
where the true solution $\varphi^{\dagger}$ 
satisfies $\cT\varphi^{\dagger} = f^{\dagger}.$
Taking absolute value of the right hand side with Cauch-Schwarz inequality
and recalling that $\norm{\cT\varphi_{\alpha(\delta)} - f^{\delta}} \leq \tau\delta$
by (\ref{discrepancy_pr_definition}) brings

\bea
\label{bregman_regularizer_upper_bound}
D_{J}(\varphi_{\alpha(\delta)} , \varphi^{\dagger}) & \leq &
\frac{\tau\delta}{\alpha(\delta)} \norm{\cT^{\ast}}\norm{\varphi_{\alpha(\delta)} - \varphi^{\dagger}}
+ \frac{\delta}{\alpha(\delta)}\norm{\cT^{\ast}}\norm{\varphi_{\alpha(\delta)} - \varphi^{\dagger}}
\nonumber\\
& = & \frac{(\tau + 1)\delta}{\alpha(\delta)} \norm{\cT^{\ast}}\norm{\varphi_{\alpha(\delta)} - \varphi^{\dagger}} .
\eea
As for the upper bound for $D_{J}^{\mbox{sym}}(\varphi_{\alpha(\delta)} , \varphi^{\dagger}),$
we adapt (\ref{symmetrical_bregman}) in the following way

\bea
D_{J}^{\mbox{sym}}(\varphi_{\alpha(\delta)} , \varphi^{\dagger})
& = & D_{J}(\varphi_{\alpha(\delta)} , \varphi^{\dagger}) + 
D_{J}(\varphi^{\dagger} , \varphi_{\alpha(\delta)})
\nonumber\\
& = & \langle \nabla J(\varphi^{\dagger}) - \nabla J(\varphi_{\alpha(\delta)}) , \varphi_{\alpha(\delta)} - \varphi^{\dagger} \rangle .
\nonumber
\eea
Again by the first order optimality conditions in (\ref{optimality_1}),
then

\bea
D_{J}^{\mbox{sym}}(\varphi_{\alpha(\delta)} , \varphi^{\dagger}) = 
\frac{1}{\alpha(\delta)} \langle \cT^{\ast}(f^{\delta} - f^{\dagger}) - \cT^{\ast}(f^{\delta} - \cT\varphi_{\alpha(\delta)}) , 
\varphi_{\alpha(\delta)} - \varphi^{\dagger} \rangle
\nonumber
\eea
We split this inner product over the term 
$\varphi_{\alpha(\delta)} - \varphi^{\dagger}$ together with
the absolute value of each part as such,

\bea
D_{J}^{\mbox{sym}}(\varphi_{\alpha(\delta)} , \varphi^{\dagger}) & \leq &
\frac{1}{\alpha(\delta)}\left\{ \vert \langle \cT^{\ast}(f^{\delta} - f^{\dagger}) , \varphi_{\alpha(\delta)} - \varphi^{\dagger} \rangle \vert
+ \frac{1}{\alpha(\delta)} \vert\langle \cT^{\ast}(f^{\delta} - \cT\varphi_{\alpha(\delta)}) , 
\varphi_{\alpha(\delta)} - \varphi^{\dagger} \rangle\vert \right\}
\nonumber\\
& \leq & \frac{1}{\alpha(\delta)} \left\{ \delta \norm{\cT^{\ast}} \norm{\varphi_{\alpha(\delta)} - \varphi^{\dagger}}
+ \norm{\cT^{\ast}} \norm{\cT\varphi_{\alpha(\delta)} - f^{\delta}} \norm{\varphi_{\alpha(\delta)} - \varphi^{\dagger}} \right\} ,
\nonumber
\eea
which is the consequence of Cauchy-Schwarz. 
Now again by the condition in (\ref{discrepancy_pr_definition})

\beq
\label{bregman_symmetrical_upper_bound}
D_{J}^{\mbox{sym}}(\varphi_{\alpha(\delta)} , \varphi^{\dagger}) \leq
\frac{1}{\alpha(\delta)} \left\{ \delta \norm{\cT^{\ast}} \norm{\varphi_{\alpha(\delta)} - \varphi^{\dagger}} +
\tau\delta \norm{\cT^{\ast}} \norm{\varphi_{\alpha(\delta)} - \varphi^{\dagger}} \right\}.
\eeq
Considering the defined regularization parameter,
$\alpha(\delta) := \sqrt{\delta} (\tau + 1) \norm{\cT^{\ast}},$
both in (\ref{bregman_regularizer_upper_bound}) and in 
(\ref{bregman_symmetrical_upper_bound}) yields the desired
upper bounds for $D_{J}(\varphi_{\alpha(\delta)} , \varphi^{\dagger})$
and $D_{J}^{\mbox{sym}}(\varphi_{\alpha(\delta)} , \varphi^{\dagger})$
respectively. Since $J(\cdot)$ is $2-$convex, then
the norm convergence of $\norm{\varphi_{\alpha(\delta)} - \varphi^{\dagger}}$
is obtained due to (\ref{q_convexity}).
\end{proof}

In fact those rates also imply another
faster convergence rate when the regularization
parameter is defined as 
$\alpha(\delta) := \sqrt{\delta} (\tau + 1) \norm{\cT^{\ast}}$. 
To observe this, different formulation of the Bregman
divergence is necessary. In the Definition
\ref{total_convexity}, take 
$\Phi(\cdot) := G_{\delta}(\cdot , f^{\delta}) = \frac{1}{2} \norm{\cT(\cdot) - f^{\delta}}^2$
to formulate the following. However, we need to recall
the assumptions about the $2-$convexity of $G_{\delta}(\cdot , f^{\delta})$
in (\ref{misfit_2-cnvx_assumption1}) and (\ref{misfit_2-cnvx_assumption2}). 

\begin{theorem}
Let $\cT : \cH \rightarrow \cH$ be a compact
forward operator in the problem (\ref{problem2}) and
assume that the conditions in 
(\ref{misfit_2-cnvx_assumption1}) and (\ref{misfit_2-cnvx_assumption2})
are satisfied. We formulate a Bregman divergence 
associated with the misfit term 
$G_{\delta}(\cdot , f^{\delta}) := \frac{1}{2}\norm{\cT(\cdot) - f^{\delta}}^2,$

\beq
\label{bregman_misfit}
D_{G}(\varphi_{\alpha(\delta)} , \varphi^{\dagger})
= \frac{1}{2}\norm{\cT\varphi_{\alpha(\delta)} - f^{\delta}}^2
- \frac{1}{2}\norm{\cT\varphi^{\dagger} - f^{\delta}}^2
- \langle \nabla G_{\delta}(\varphi^{\dagger} , f^{\delta}) , 
\varphi_{\alpha(\delta)} - \varphi^{\dagger} \rangle .
\eeq
If $\varphi_{\alpha(\delta)}$ is the regularized minima
for the problem (\ref{problem2}), then with the choice
of regularization parameter
$\alpha(\delta) := \sqrt{\delta} (\tau + 1) \norm{\cT^{\ast}}$
for sufficiently small $\delta \in (0,1),$

\beq
D_{G_{\delta}}(\varphi_{\alpha(\delta)} , \varphi^{\dagger})
\leq \cO(\delta^{3/2})
\eeq
As expected, this rate also implies the following

\beq
\norm{\varphi_{\alpha(\delta)} - \varphi^{\dagger}} \leq \cO(\delta^{3/4}) .
\eeq
\end{theorem}

\begin{proof}

As given by (\ref{discrepancy_pr_definition}),
$\norm{\cT\varphi_{\alpha(\delta)} - f^{\delta}} \leq \tau\delta.$
Additionally the noisy measurement $f^{\delta}$
to the true measurement $f^{\dagger}$ satisfies
$\norm{f^{\delta} - f^{\dagger}} \leq \delta.$
In the Theorem 
\ref{Bregman_regularizer_upper_bound} 
above, we have estimated
a pair of convergence rates with the same 
regularization parameter $\alpha(\delta).$
So for 
$D_{G_{\delta}}(\varphi_{\alpha(\delta)} , \varphi^{\dagger})$
defined by (\ref{bregman_misfit}) will provide the result below;

\bea
D_{G_{\delta}}(\varphi_{\alpha(\delta)} , \varphi^{\dagger})
& \leq & \frac{1}{2}(\tau\delta)^2 + \frac{1}{2}\delta^2 - 
\langle \cT^{\ast}(f^{\dagger} - f^{\delta}) , \varphi_{\alpha(\delta)} - \varphi^{\dagger} \rangle
\nonumber\\
& = & \frac{1}{2}(\tau\delta)^2 + \frac{1}{2}\delta^2 - 
\langle f^{\dagger} - f^{\delta} , \cT^{\ast}(\varphi_{\alpha(\delta)} - \varphi^{\dagger}) \rangle
\nonumber\\
& \leq &
\frac{1}{2}\delta^2(\tau^2 + 1) + \delta \norm{\cT^{\ast}} \norm{\varphi_{\alpha(\delta)} - \varphi^{\dagger}}
\nonumber
\eea
As has been estimated in the Theorem \ref{Bregman_regularizer_upper_bound}
$\norm{\varphi_{\alpha(\delta)} - \varphi^{\dagger}} \leq \sqrt{\delta}$
when $\alpha(\delta) := \sqrt{\delta} (\tau + 1) \norm{\cT^{\ast}}.$
Hence,

\bea
D_{G_{\delta}}(\varphi_{\alpha(\delta)} , \varphi^{\dagger})
& \leq & \frac{1}{2}\delta^2(\tau^2 + 1) + \delta^{3/2}\norm{\cT^{\ast}}
\nonumber\\
& \leq & \delta^{3/2}\left( \frac{1}{2}(\tau^2 + 1) + \norm{\cT^{\ast}} \right).
\eea
Now, since $G_{\delta}(\cdot , f^{\delta})$ is $2-$convex (see Def. \ref{total_convexity}),
by (\ref{q_convexity}) and by the assumptions 
(\ref{misfit_2-cnvx_assumption1}) and (\ref{misfit_2-cnvx_assumption2}), 
we have,

\bea
\norm{\varphi_{\alpha(\delta)} - \varphi^{\dagger}} \leq 
\frac{\delta^{3/4}}{c_{f^{\delta}}}
\left( \frac{1}{2}(\tau^2 + 1) + \norm{\cT^{\ast}} \right)^{1/2} .
\eea
\end{proof}

\subsection{When the penalty $J(\cdot)$ is defined over $\cC^{2+}(\Omega,\cH)$}

Surely the convergence rates above are still preserved when
the penalty $J(\cdot)$ is defined over $\cC^{2+}(\Omega,\cH)$
since $\cC^{2} \subset \cC^{1}.$
However, one may be interested in discrepancy principle
in this more specific case.
Above, we have formulated those convergence rates under
the assumption $J(\cdot) : \cC^{1}(\Omega,\cH) \rightarrow \R_{+}.$
We will now analyse the convergence with assuming
$J(\cdot) : \cC^{2+}(\Omega,\cH) \rightarrow \R_{+}.$
Here we will define regularization parameter also as a function
of Hessian Lipschitz constant $L_H$, \cite{FowkesGouldFarmer13}.
We begin with estimating the dicrepancy 
$\norm{\cT\varphi_{\alpha(\delta)} - f^{\delta}}_2.$


\begin{theorem}
Let $F_{\alpha}(\cdot)$ be the smooth and convex cost functional 
as defined in the problem (\ref{problem2}).
If the penalty $J(\cdot) : \cC^{2+}(\Omega,\cH) \rightarrow \R_{+}$ 
is strongly convex, then 

\bea
\norm{\cT\varphi_{\alpha(\delta)} - f^{\delta}} \leq \delta \left(1 + \frac{1}{L_H}\norm{\cT^{\ast}}^2 \right)^{1/2}
+ \sqrt{\tilde{\cO}(\norm{\varphi_{\alpha(\delta)} - \varphi^{\dagger}}^2)}
\nonumber
\eea
where $L_H$ is the Hessian Lipschitz constant of the functional $F_{\alpha}(\cdot).$
\end{theorem}

\begin{proof}

Let us consider the following second order Taylor expansion,

\bea
F_{\alpha}(\varphi_{\alpha(\delta)}) = F_{\alpha}(\varphi^{\dagger}) & + &
\langle \varphi_{\alpha(\delta)} - \varphi^{\dagger} , \nabla F_{\alpha}(\varphi^{\dagger}) \rangle +
\nonumber\\
& + & \frac{1}{2} \langle \varphi_{\alpha(\delta)} - \varphi^{\dagger} , 
\nabla^2 F_{\alpha}(\varphi^{\dagger})(\varphi_{\alpha(\delta)} - \varphi^{\dagger}) \rangle
+ \cO(\norm{\varphi_{\alpha(\delta)} - \varphi^{\dagger}}^2)
\nonumber
\eea
Obviously, this Taylor expansion is bounded by
\bea
F_{\alpha}(\varphi_{\alpha(\delta)}) \leq F_{\alpha}(\varphi^{\dagger}) + 
\langle \varphi_{\alpha(\delta)} - \varphi^{\dagger} , \nabla F(\varphi^{\dagger}) \rangle +
\frac{1}{2} L_{H} \norm{\varphi_{\alpha(\delta)} - \varphi^{\dagger}}^2
+ \cO(\norm{\varphi_{\alpha(\delta)} - \varphi^{\dagger}}^2) ,
\nonumber
\eea
where $L_H$ is the Hessian Lipschitz constant of the functional $F_{\alpha}(\cdot).$
After some arrangement with
the explicit definition $F_{\alpha}(\cdot)$ in the problem 
(\ref{problem2}) the inequality above reads,

\bea
\frac{1}{2}\norm{\cT\varphi_{\alpha(\delta)} - f^{\delta}}^2 \leq \frac{1}{2}\delta^2
+\alpha(\delta) \left(J(\varphi^{\dagger}) - J(\varphi_{\alpha(\delta)})\right) +  
\nonumber
\eea
\bea
+ \alpha(\delta)\langle \varphi_{\alpha(\delta)} - \varphi^{\dagger} , \nabla J(\varphi^{\dagger}) \rangle +
\langle \varphi_{\alpha(\delta)} - \varphi^{\dagger} , \cT^{\ast}(f^{\dagger} - f^{\delta}) \rangle & + &
\frac{1}{2} L_{H} \norm{\varphi_{\alpha(\delta)} - \varphi^{\dagger}}^2
\nonumber\\
& + & \cO(\norm{\varphi_{\alpha(\delta)} - \varphi^{\dagger}}^2) .
\nonumber
\eea
Now by the early estimations for the difference $J(\varphi^{\dagger}) - J(\varphi_{\alpha(\delta)})$
in (\ref{upper_boundary}),
\bea
\frac{1}{2}\norm{\cT\varphi_{\alpha(\delta)} - f^{\delta}}^2 
\leq \frac{1}{2}\delta^2 + \alpha(\delta) \langle \nabla J(\varphi^{\dagger}) , \varphi^{\dagger} - \varphi_{\alpha(\delta)} \rangle 
- \alpha(\delta) \langle \nabla J(\varphi^{\dagger}) , \varphi^{\dagger} - \varphi_{\alpha(\delta)} \rangle + 
\nonumber
\eea
\bea
+ \langle \varphi_{\alpha(\delta)} - \varphi^{\dagger} , \cT^{\ast}(f^{\dagger} - f^{\delta}) \rangle
+ \frac{1}{2} L_{H} \norm{\varphi_{\alpha(\delta)} - \varphi^{\dagger}}^2 
+ \cO(\norm{\varphi_{\alpha(\delta)} - \varphi^{\dagger}}^2) .
\nonumber
\eea
After Cauchy-Schwarz and Young's inequalities on the right hand side, 
we have

\bea
\frac{1}{2}\norm{\cT\varphi_{\alpha(\delta)} - f^{\delta}}^2 & \leq & \frac{1}{2}\delta^2 +
\delta \norm{\cT} \norm{\varphi_{\alpha(\delta)} - \varphi^{\dagger}}
+ \frac{1}{2} L_{H} \norm{\varphi_{\alpha(\delta)} - \varphi^{\dagger}}^2 .
\nonumber\\
& \leq ^{\footnotemark} & \delta^2 \left( \frac{1}{2} + \frac{1}{2 L_H} \norm{\cT}^2 \right) + 
L_H \norm{\varphi_{\alpha(\delta)} - \varphi^{\dagger}}^2 +
\cO(\norm{\varphi_{\alpha(\delta)} - \varphi^{\dagger}}^2) .
\nonumber
\eea
\footnotetext{ For some $\epsilon > 0,$ by Young's inequality 
$\delta \norm{\cT^{\ast}} \norm{\varphi_{\alpha(\delta)} - \varphi^{\dagger}}\leq \frac{\epsilon}{2}\delta^2 \norm{\cT^{\ast}}^2
+ \frac{1}{2\epsilon}\norm{\varphi_{\alpha(\delta)} - \varphi^{\dagger}}^2.$ If we take 
$\epsilon = 1/L_H,$ then the inequality follows.}
In the name of convenience, we combine the last two terms
on the right hand side under one notation $\tilde{\cO}$. 
Then,

\bea
\frac{1}{2}\norm{\cT\varphi_{\alpha(\delta)} - f^{\delta}}^2
\leq \delta^2 \left( \frac{1}{2} + \frac{1}{2 L_H} \norm{\cT}^2 \right)
+ \tilde{\cO}(\norm{\varphi_{\alpha(\delta)} - \varphi^{\dagger}}^2) .
\nonumber
\eea
Since $\left( \sqrt{a} + \sqrt{b} \right)^2 = a + b + 2 \sqrt{a b} \geq a + b$ 
for $a, b \in \R_{+},$ hence

\beq
\norm{\cT\varphi_{\alpha(\delta)} - f^{\delta}} \leq \delta \left(1 + \frac{1}{L_H}\norm{\cT^{\ast}}^2 \right)^{1/2}
+ \sqrt{\tilde{\cO}(\norm{\varphi_{\alpha(\delta)} - \varphi^{\dagger}}^2)} .
\eeq

\end{proof}

\begin{remark}[The coefficient $\tau$ in the limit sense]
In the theorem, the remaining term is 
$\sqrt{\tilde{\cO}(\norm{\varphi_{\alpha(\delta)} - \varphi^{\dagger}}^2)}.$
As a result of any regularization strategy, it is expected that
$\norm{\varphi_{\alpha(\delta)} - \varphi^{\dagger}} \rightarrow 0$
as $\alpha(\delta) \rightarrow 0.$ Hence in the limit sense,
the coefficient $\tau$ in (\ref{discrepancy_pr_definition})
may be defined as

\beq
\tau(L_H) := \left(1 + \frac{1}{L_H}\norm{\cT^{\ast}}^2 \right)^{1/2} .
\eeq

\end{remark}

\begin{remark}[Preservation of the convergence rates]
Owing to the Theorem \ref{Bregman_regularizer_upper_bound},
it is easy to conclude that the convergence rates defined
above are preserved when the penalty $J(\cdot)$ 
is $2-$convexity by (\ref{q_convexity}) and 
the regularization parameter is defined as,

\beq
\alpha(\delta) := \sqrt{\delta} (\tau(L_H) + 1) \norm{\cT^{\ast}} ,
\eeq
where $\tau(L_H) := \left(1 + \frac{1}{L_H}\norm{\cT^{\ast}}^2 \right)^{1/2}.$
\end{remark}


\section{Summary of the Convergence Rates}
\label{table_convergence_rates}

In this work, we have obtained the convergence rates with following the 
footsteps of the counterpart works in 
\cite{BurgerOsher04, Grasmair10, Grasmair13}.
However, we have also taken into account one more fact which is
$\norm{\cT \varphi_{\alpha(\delta)} - f^{\delta}} \leq \tau \delta$
where $\alpha(\delta)$ fulfils the condition (\ref{discrepancy_pr_definition}).
It has been observed that $2-$convexity condition for the penalty $J(\cdot)$
is crucial to obtain norm covergence by means of Bregman divergence.
We have not given any analytical evaluation of $\tau$ without
any specific penalty $J(\cdot).$ Note that these convergence
rates are true for $J(\cdot) \in \cC^{k}(\Omega,\cH)$
where $k = 1$ and $k = 2.$ Below we summarize these 
corresponding convergence rate estimations per Bregman divergence 
formulation.  

\begin{center}
    \begin{tabular}{ | l | l | l | p{5cm} |}
    \hline
    $\alpha(\delta)$ & Bregman divergence estimate & $\norm{\varphi_{\alpha(\delta)} - \varphi^{\dagger}}$ estimate \\ \hline
    $\alpha(\delta) := \sqrt{\delta} (\tau + 1) \norm{\cT^{\ast}}$ & $D_{J}(\varphi_{\alpha(\delta)} , \varphi^{\dagger}) 
    \leq\cO(\sqrt{\delta}),$ & $\sqrt{\delta}$ \\ \hline
    $\alpha(\delta) := \sqrt{\delta} (\tau + 1) \norm{\cT^{\ast}}$ & $D_{J}^{\mbox{sym}}(\varphi_{\alpha(\delta)} , \varphi^{\dagger}) 
    \leq\cO(\sqrt{\delta}),$ & $\sqrt{\delta}$ \\ \hline
    $\alpha(\delta) := \sqrt{\delta} (\tau + 1) \norm{\cT^{\ast}}$ & $D_{G}(\varphi_{\alpha(\delta)} , \varphi^{\dagger}) 
    \leq\cO(\delta^{3/2}),$ & $\frac{\delta^{3/4}}{c_{f^{\delta}}}\left( \frac{1}{2}(\tau^2 + 1) + \norm{\cT^{\ast}} \right)^{1/2}$ \\ \hline
    \end{tabular}
\end{center}

\section*{Acknowledgement}

The author is indepted to Prof. Dr. D. Russell Luke for valuable discussions 
on different parts of this work.

\newpage



\begin{appendices}
\section{Further properties of the Bregman divergence} 
\label{more_Bregman_divergence}

Although $D_{F}(\varphi_{\alpha(\delta)} , \varphi^{\dagger})$
has been introduced above in Definition \ref{various_bregman}
by (\ref{bregman_def_cost_func}), 
an immediate conclusion can be formulated below.

\begin{cor}
If $\varphi_{\alpha(\delta)}$ and $\varphi^{\dagger}$ 
are the regularized and the true solutions respectively
to the problem (\ref{problem2}) wherein the cost functional
$F$ is convex and smooth, then

\beq
D_{F}(\varphi_{\alpha(\delta)} , \varphi^{\dagger}) = 0 .
\eeq
\end{cor}

\begin{proof}
Since $\varphi_{\alpha(\delta)}$ is the minimizer,
then $F(\varphi_{\alpha(\delta)}) \leq F(\varphi)$
for all $\varphi \in \Omega,$ which implies 

\beq
D_{F}(\varphi_{\alpha(\delta)} , \varphi^{\dagger})
\leq - \langle \nabla F(\varphi^{\dagger}) , \varphi_{\alpha(\delta)} - \varphi^{\dagger} \rangle  .
\eeq
On the other hand, just by logic, $\nabla F(\varphi^{\dagger}) = 0.$
It is known that, for any convex functional $\Phi$
the Bregman divergence $D_{\Phi} \geq 0.$ Hence
$D_{F}(\varphi_{\alpha(\delta)} , \varphi^{\dagger}) = 0 .$

\end{proof}

Addition to this, a relation between
$D_{J}^{\mbox{sym}}(\varphi_{\alpha(\delta)} , \varphi^{\dagger})$
and 
$D_{G}^{\mbox{sym}}(\varphi_{\alpha(\delta)} , \varphi^{\dagger})$
can also be observed.

\begin{theorem}
Let the regularized minimum $\varphi_{\alpha(\delta)}$
to the problem (\ref{problem2}) satisfy the first order 
optimality conditions (\ref{optimality_1}).
Then the following inclusion holds true for $\alpha > 0,$

\beq
\alpha D_{J}^{\mbox{sym}}(\varphi_{\alpha(\delta)} , \varphi^{\dagger}) = 
D_{G}^{\mbox{sym}}(\varphi_{\alpha(\delta)} , \varphi^{\dagger}) .
\eeq
\end{theorem}

\begin{proof}
As defined by (\ref{bregman_def_misfit}), one can directly
derive

\bea
D_{G}^{\mbox{sym}}(\varphi_{\alpha(\delta)} , \varphi^{\dagger})
& = & \langle \nabla G(\varphi_{\alpha}^{\delta}) - \nabla G(\varphi^{\dagger}) , 
\varphi_{\alpha}^{\delta} - \varphi^{\dagger}\rangle 
\nonumber\\
& = &\langle \cT^{\ast}(\cT\varphi_{\alpha}^{\delta} - f^{\delta}) - 
\cT^{\ast}(f^{\dagger} - f^{\delta}) , \varphi_{\alpha}^{\delta} - \varphi^{\dagger} \rangle
\eea
In proof of Theorem \ref{Bregman_regularizer_upper_bound}, 
or by (\ref{symmetrical_bregman}),
$D_{J}^{\mbox{sym}}(\varphi_{\alpha(\delta)} , \varphi^{\dagger})$
has been given already. Since $\varphi_{\alpha}^{\delta}$
satisfies the first order optimality conditions (\ref{optimality_1}),

\beq
D_{J}^{\mbox{sym}}(\varphi_{\alpha(\delta)} , \varphi^{\dagger})
= \frac{1}{\alpha} \langle \cT^{\ast}(\cT\varphi_{\alpha}^{\delta} - f^{\delta}) - 
\cT^{\ast}(f^{\dagger} - f^{\delta}) , \varphi_{\alpha}^{\delta} - \varphi^{\dagger} \rangle
\eeq
which yields the result.
\end{proof}

\end{appendices}

\bigskip
\section*{References}


 \bibliographystyle{alpha}

\end{document}